\documentclass[12pt,leqno]{amsart}

\usepackage{amssymb,amsfonts,amsmath,amsthm}
\usepackage{latexsym}
\usepackage{verbatim}
\usepackage{epsfig}
\usepackage[active]{srcltx}
\usepackage[hypertex]{hyperref}

\usepackage[T1]{fontenc}
\usepackage[latin1]{inputenc}

\newtheorem{theorem}{Theorem}[section]
\newtheorem{conjecture}[theorem]{Conjecture}
\newtheorem{corollary}[theorem]{Corollary}
\newtheorem{lemma}[theorem]{Lemma}

\theoremstyle{definition}

\theoremstyle{remark}

\renewcommand{\Box}{\square}    

\newcommand{\h}{{\rm{ht}}}
\newcommand{\homeo}{{\rm{homeo}}}

\newcommand{\lk}{\mathop{\rm{lk}}\nolimits}
\newcommand{\Cl}{\mathop{{\bC}\rm{lk}}\nolimits}

\newcommand{\Sing}{{\rm{Sing\hspace{2pt}}}}
\newcommand{\rank}{{\rm{rank\hspace{2pt}}}}

\newcommand{\im}{{\rm{Im}}}

\newcommand{\grad}{\mathop{\rm{grad}}}

\newcommand{\e}{\varepsilon}
\newcommand{\m}{\setminus}
\newcommand{\fin}{\hspace*{\fill}$\Box$\vspace*{2mm}}

\newcommand{\bR}{{\mathbb R}}
\newcommand{\bC}{{\mathbb C}}

\newcommand{\bZ}{{\mathbb Z}}

\begin{document}

\title{Beyond Mumford's theorem on normal surfaces}

\author{Mihai Tib\u ar} 

\address{Math\' ematiques, UMR 8524 CNRS,
Universit\'e de  Lille 1, \  59655 Villeneuve d'Ascq, France.}

\email{tibar@math.univ-lille1.fr}

\dedicatory{To Mihnea Col\c toiu, on the occasion of a great  anniversary}




\begin{abstract}

 Beyond normal surfaces there are several open questions concerning 2-dimensional spaces. We  present some results and conjectures along this line.  

\end{abstract}

\maketitle

\setcounter{section}{0}


\section{Questions}

Let $(X, x_0)$ be the germ of a surface. Let us assume that $(X, x_0)$ has an isolated singularity at the origin and moreover that $(X, x_0)$ is \textit{normal}.  Mumford has proved  that the ``triviality'' of the singular germ may be translated into the ``simplicity'' of its link: 

\begin{theorem}\label{t:mumford} \rm (Mumford, 1961) \it
The link of a normal surface is simply connected if and only if the surface is non-singular.
\end{theorem}

The ``simplicity'' of a surface is understood here as the simplicity of its link.
We recall that the \textit{link} (more precisely, \textit{real link}) of an analytic set $Y$ at some point $y\in Y$ is by definition $\lk_\e(Y,y) := Y\cap \rho^{-1}(\e)$, where  
$\rho : Y \to \bR_{\ge 0}$ is  a non-trivial analytic function such that $\rho^{-1}(0) = \{ y\}$. By the \textit{local conic structure} of analytic sets, result due to Burghelea and Verona \cite{BV}, the link $\lk_\e(Y,y)$ does not depend, up to homeomorphisms, neither on the choice of $\rho$, nor on that of $\e >0$ provided that it is small enough. For instance $\rho$ may be the distance function to the point $y$. 

\noindent
We address here two situations beyond Mumford's setting:

\begin{enumerate}
\item $(X, x_0)$  normal with nontrivial  $\pi_1 (\lk(X, x_0))$. \\
\item $(X, x_0)$ not normal but $\pi_1(\lk(X, x_0))$ trivial. 
\end{enumerate}

\noindent
The issue (a) is discussed in \S\ref{s:1}. By Prill's result and by the more recent one by Mihnea Col\c toiu and the author,  normal surfaces fall into precisely two disjoint classes: the universal covering of the punctured germ $X \setminus \{ x_0\}$ is a Stein manifold or not. The former corresponds to infinite $\pi_1(\lk(X, x_0))$ and the later to finite $\pi_1(\lk(X, x_0))$.

In sections \S 3-\S\ref{s:nonisol} we discuss  several aspects of the issue (b). It turns out that Mumford's theorem does not extend to surfaces with isolated singularities, as we show by an  example in \S \ref{s:example}. In case of non-isolated singularities,  for surfaces in $\bC^3$ we prove a criterion of simplicity in terms of the triviality of the complex link. This also yields an approach to the following ``rank conjecture'': \emph{an injective map germ $g : (\bC^2, 0) \to (\bC^3, 0)$ has  $\rank (\grad g) (0) \ge 1$.}

\section{Steinness of the universal coverings}\label{s:1}

We assume in this section that $(X, x_0)$ is a normal surface. One has the following classical result by D. Prill.

\begin{theorem}\rm \cite{Pr} \it

 If the fundamental group of the link is finite then the surface
$(X, x_0)$ is isomorphic to the quotient of $(\bC^2, 0)$ by a finite linear group.
\end{theorem} 

In particular this shows that the universal covering of the complement $X\setminus \{x_0\}$ is 
 not Stein. One may say that $X$ is of ``concave'' type, whereas the ``convex'' type surfaces would be those for which the universal covering of $X \setminus \{ x_0\}$ is a Stein manifold. 
  It turns out that all non-quotient surface singularities are in fact of convex type:

\begin{theorem}\label{t:coltoiu1} \rm \cite{CT} \it \\
If $\pi_1(\lk(X, x_0))$ is infinite, then the universal covering of $X \setminus \{ x_0\}$ is a Stein manifold.
\end{theorem}

The proof uses the structure of a resolution of the normal surface singularity $p : Y \to X$ in the neighbourhood of its exceptional divisor $ A:= p^{-1}(x_0)$ with normal crossings. 

There are two cases to treat: $H_1(\lk(X, x_0)$ is infinite or finite. Starting from Mumford's result that the non-finiteness of $H_1(\lk(X, x_0)$ is equivalent to that of $\pi_1(A)$ one constructs special coverings of $A$ and of its neighbourhood such that the former are Stein. There are mainly 3 cases: the dual graph of $A$ is a cycle, contains a cycle, or does not contain. In order to settle the problem, one uses several results on Stein spaces, Runge pairs and other related properties, due to Napier \cite{Nap}, Nori, Gurjar, Coltoiu \cite{Co1}, Simha \cite{Sim} and others.

 The case of finite homology group $H_1(\lk(X, x_0)$ is reduced to the previous case by topological methods involving the graph manifold structure of the link (or Waldhausen structure \cite{Wa}) and its Jaco-Shalen-Johannson decomposition \cite{JS}, \cite{Jo}. Namely one shows that there is a finite covering of the link with infinite $H_1$ and moreover one may attach to it a normal surface germ $Z$ and an analytic covering $Z\to X$.
 
 Finally, if one constructs a covering of the complement which is Stein, then all the coverings of it are Stein (by K. Stein's theorem \cite{St}) and in particular the universal one is Stein.

The above result \cite{CT}  has been extended recently by Col\c toiu, Joi\c ta and the author to higher dimensions \cite{CJT}.

\section{Simplicity of surfaces via the simplicity of their complex links}

One may ask if there is any ``simplicity'' criterion beyond the class of normal surfaces.
We consider here this question for both types of non-normal surface germs: surfaces with isolated singularity and surfaces with non-isolated singularities.

We first observe that 
for normal surfaces Mumford's criterion is also equivalent to the ``simplicity'' of the \emph{complex link} $\Cl(X,0)$, more precisely,  that \emph{a normal surface germ is non-singular if and only if its complex link is contractible} (Theorem \ref{c:mum}).
Using this new criterion (more precisely, Lemma \ref{t:mum2} below) we show by an  example in \S \ref{s:example}  that Mumford's criterion does not extend to isolated singularities.

Following the same idea of replacing the real link condition by a complex link condition, we show the following extension of Mumford's criterion to non-isolated singularities. This can be be compared to an observation by L\^e D.T in \cite[Rem. 2, pag. 285]{Le2}.

\begin{theorem}\label{t:clink}
Let $(X, 0)\subset (\bC^3, 0)$ be a surface germ with 1-dimensional singularity.
 The complex link $\Cl(X,0)$ is contractible if and only if  $X$ is an equisingular deformation of a curve.
 \end{theorem}

We mean by ``$X$ is an equisingular deformation of a curve'' that $X$ is the total space of a one-parameter family of curves $X_t$ with a single singularity and having the same Milnor number.   
 It does not follow that the curves are irreducible, as shown by the following example:
 $X := \{ xy =0 \} \times \bC \subset \bC^3$ is a surface with contractible complex link
 (see \S \ref{ss:link} for the definition) and a trivial family of reducible curves. However, if we impose in addition that the real link is homeomorphic to $S^3$ then 
 the irreducibility follows too. It turns out that this criterion has an extension to  higher dimensions (Corollary \ref{c:hyp}).

Under the same exchange between the real link condition and the complex link condition, our Theorem \ref{t:clink} corresponds to the following conjecture due to L\^e D\~ung Tr\' ang which holds since about 30 years.  
\begin{conjecture}\label{conj1}
 Let $(X, 0) \subset (\bC^3, 0)$ be a surface with 1-dimensional singular locus.
  Then the real link of $(X, 0)$ is simply connected if and only if  $(X, 0)$ is an equisingular deformation of an irreducible curve.
\end{conjecture}
 Several particular classes of surfaces have been checked out confirming the conjecture, see e.g. \cite{Fe}.  L\^e D.T. observed that this conjecture can be equivalently formulated as follows.

 \begin{conjecture}\label{conj2}
 An injective map germ $g : (\bC^2, 0) \to (\bC^3, 0)$ has  $\rank (\grad g) (0) \ge 1$.
\end{conjecture}
The latter looks like a more elementary statement but appears to be very delicate to prove even in particular cases, see \cite{Ne} and \cite{KM}.
  Our Theorem \ref{t:clink} yields in particular the following equivalent formulation of Conjecture \ref{conj1}: 
 
 \begin{conjecture}\label{conj3}
 Let $(X, 0) \subset (\bC^3, 0)$ be a surface with 1-dimensional singular locus.
  If the real link $\lk (X, 0)$ is simply connected then the complex link $\Cl(X, 0)$ is contractible.
\end{conjecture}


\section{Simplicity of surfaces via the simplicity of their complex links}

\subsection{A supplement to Mumford's criterion}\label{ss:link}
We recall that the \textit{link} (more precisely, \textit{real link}) of an analytic set $Y$ at some point $y\in Y$ is by definition $\lk_\e(Y,y) :=  \rho^{-1}(\e)$, where  
$\rho : (Y, y) \to (\bR_{\ge 0}, 0)$ is  a proper analytic function such that $\rho^{-1}(0) = \{ y\}$. By the \textit{local conic structure} of analytic sets, result due to Burghelea and Verona \cite{BV}, the link $\lk_\e(Y,y)$ does not depend, up to homeomorphisms, neither on the choice of $\rho$, nor on that of $\e >0$ provided that it is small enough. For instance $\rho$ may be the distance function to the point $y$.
Given an analytic space germ $(Y, 0)$,  any holomorphic function $g : (Y,0)\to (\bC, 0)$ defines a local \textit{Milnor-L\^e fibration} \cite{Le}, which means that :
\begin{equation}\label{eq:loctriv}
 g_| : B_\varepsilon \cap g^{-1}(D_\delta^*) \to D_\delta^*
\end{equation}
  is a locally trivial fibration for small enough $\varepsilon \gg \delta >0$. 
  
  Moreover,  the restriction:
\begin{equation}\label{eq:triv}
 g_| : \partial B_\varepsilon \cap g^{-1}(D_\delta) \to D_\delta
\end{equation}
is a trivial fibration. Then $g^{-1}(\lambda) \cap B_\e$, for $0< |\lambda| \ll \e$, is called \emph{Milnor fibre}.
 
The \textit{complex link} of $(Y, 0)$, denoted by $\Cl(Y,0)$,  is defined as the Milnor fibre of  
 a holomorphic \textit{general} function. 
This notion has been introduced in the 1980's by Goresky and MacPherson \cite{GM} who proved that its homotopy type does not depend on the choices of $\lambda$ and $\e$ provided they are small enough.

 The complex link of a non-singular space germ $(Y,0)$ is clearly contractible, whereas the converse turns out to be not true in general (see e.g. the example of \S \ref{s:example}). The  following result is due to Looijenga \cite[pag. 68]{Lo}: 
\emph{If $(Y,0)$ is a complete intersection with isolated singularity
then the complex link of $(Y,0)$ is contractible if and only if $(Y,0)$ nonsigular.}

We ad up one more equivalence to Mumford's result, which shows that,
in the quest for ``simplicity'', the complex link plays a similar role as the real link.  

 \begin{theorem}\label{c:mum}\rm(A supplement to Mumford's theorem)\it  \\
 Let $(X,0)$ be an irreducible normal surface germ. The following are equivalent:
 \begin{enumerate}
 \item $X$ is non-singular at 0.
 \item $\lk(X, 0)$ is simply connected.
 \item $\Cl(X,0)$ is contractible.
  \end{enumerate}
\end{theorem}

\subsection{Proof of Theorem \ref{c:mum}} 
A smooth surface has contractible complex link, and its link is simply connected, thus (a) $\Leftrightarrow$ (c) and (a) $\Leftrightarrow$ (b) are obvious. The implication (b) $\Leftrightarrow$ (a) is of course Mumford's theorem.
 The implication (c) $\Rightarrow$ (b) follows from the next lemma, which assumes ``isolated singularity'' but not necessarily ``normal''.\footnote{We do not know whether the converse of Lemma \ref{t:mum2}  is true or not.}

\begin{lemma}\label{t:mum2}
Let $(X,0)$ be a surface germ with isolated singularity. If $\Cl(X,0)$ is contractible 
then $\lk(X, 0)$ is homeomorphic to $S^3$.


  \end{lemma}
  
\begin{proof}
Let $g$ be a generic function on the surface $(X,0)$. We refer to the Milnor fibration \eqref{eq:loctriv} of the function $g$ and we use the notation $X_M :=g^{-1}(M) \cap B_\e$ for some subset $M\subset D_\delta$.
 By hypothesis, the fibre
   $X_\lambda$ is contractible, $\lambda\in D_\delta\setminus \{0\}$. Then 
its boundary $\partial X_\lambda := \partial B_\e \cap  X_\lambda$ is a circle $S^1$.
 Indeed, by \eqref{eq:triv} we have that $\partial X_\lambda$ is diffeomorphic to the link of the function $g$, hence it is a disjoint union of circles. It is therefore enough to show that there is a single circle. From the homology exact sequence of the pair $(\bar X_\lambda, \partial X_\lambda)$ and since
$H_1(\bar X_\lambda, \partial X_\lambda)$ is dual to $H^1(X_\lambda)$ thus isomorphic to $H_1(X_\lambda)$,
it follows that $H_0(\partial X_\lambda)= \bZ$. This shows our claim.

For the link $\lk(X,0) := \partial B_\e \cap X$ we have the homotopy equivalence:
\[ \lk(X,0) \stackrel{\h}{\simeq} (B_\e \cap X_{\partial D_\delta}) \cup_{\partial B_\e \cap X_{\partial D_\delta} } \partial B_\e \cap X_{D_\delta}.
 \]
 By (\ref{eq:loctriv}), the space $B_\e \cap X_{\partial D_\delta}$
is a locally trivial fibration over $\partial D_\delta$ with contractible fibre (the complex link) and it is therefore homeomorphic to a full torus $S^1\times D_\delta$. 

Due to (\ref{eq:triv}), the term $\partial B_\e \cap X_{D_\delta}$  is a trivial fibration 
over $D_\delta$  with fibre $\partial X_0 \simeq S^1$, hence homeomorphic to a solid torus $D_\delta\times S^1$. The common boundary $\partial B_\e \cap X_{\partial D_\delta} $ is a trivial fibration over $\partial D_\delta$ with  fibre $S^1$ as shown above, hence homeomorphic to a torus $S^1 \times S^1$. The glueing of the two tori produces a lens space and since this glueing is trivial by the fact that $S^1 \times S^1$ has trivial monodromy, it follows that this lens space is a 3-sphere.

This shows the homotopy equivalence $\lk(X,0)\stackrel{\h}{\simeq} S^3$, hence that
$\lk (X, 0)$ is simply connected. 
\end{proof}

\section{Mumford's criterion does not extend to surfaces with isolated singularity}\label{s:example}

  Let us consider the following low degree example $F : \bC^2 \to \bC^4$, $F(x,y) = 
(x, y^2, xy^3, xy + y^3)$. Then $F$ is injective and proper, with $Y = \im F$ a surface which is not a complete intersection. Since our space $Y$ has an isolated singularity, a function 
$g : (Y, 0) \to \bC$ has an isolated singularity if and only if $g$ has no singularities on  $Y \setminus \{ 0\}$.

  By the general bouquet theorem for functions $(X,0)\to \bC$ with isolated singularities with respect to some Whitney stratification of $X$ \cite{Ti-b}, the Milnor fibre of such a function has the homotopy type of a bouquet  of spaces out of which one is $\Cl (X,0)$. This 
yields a ``minimality'' property of complex links, i.e.
\cite[Corollary 2.6]{Ti-min}:  \it
The complex link $\Cl (X,0)$ is minimal among the Milnor fibres of functions $(X,0)\to \bC$ with isolated singularity.\rm

Consequently, in order to show that the complex link of our space $(Y,0)$ is contractible it is enough to exhibit a function on $Y$ with isolated singularity and with contractible Milnor fibre.
If $z_1$ denotes the first coordinate of $\bC^4$ then consider the restriction ${z_1}_{|Y}$ of the linear function $z_1$. This has isolated singularity on $Y$. Its pull-back by $F$ is the linear function $x$ on  $\bC^2$, which has  a trivial complex link. Since the two Milnor fibres, of ${z_1}_{|Y}$ and of $x$, are homeomorphic by $F$, we have proved via \cite[Theorem 1.1]{Ti-min} that the  complex link of $(Y,0)$ is contractible.
By Lemma \ref{t:mum2}, the link  $\lk (Y, 0)$ is simply connected.  This example shows that Mumford's criterion cannot be extended to surfaces with isolated singularity. 


\section{Surfaces with non-isolated singularity}\label{s:nonisol}

\begin{proof}[\bf Proof of Theorem \ref{t:clink}]
Let $(X,0) \subset (\bC^3, 0)$ be a hypersurface germ defined by $f=0$.

 ``$\Leftarrow$''. If $X$ is an equisingular family of curves, then let $l : X\cap  B_\e \cap l^{-1}(D_\delta) \to D_\delta$ be the projection of the family, for a small enough ball $B_\e$ centered at $0$, and let $X_t := l^{-1}(t)$. 
  Consider the germ of the \emph{polar curve}.   $\Gamma(l,f) := \overline{\Sing(l,f) \m \Sing f}$ of the map $(l,f) : \bC^3 \to \bC^2$. The polar curve of a function $f$ with respect to a linear function $l$ is a old geometric notion which was brought into light in the 1970's by L\^e D.T., see e.g. \cite{Le1}, \cite{LR}, and used ever since by many authors. 
  
 By the equisingularity assumption, the curve $X_t$ has a single singularity of Milnor number $\mu(X_t)$ independent on $t$. This implies that $\Sing X$ has multiplicity equal to $1$ hence it is a non-singular irreducible curve. By \cite{LR} (see also \cite{Le1})  the invariance of $\mu(X_t)$ implies that $\Gamma(l,f) =\emptyset$.  Further on, these imply that $X_t$ is contractible. Then by \cite[Corollary 2.6]{Ti-min} the complex link of $(X,0)$ is also contractible.

\smallskip
\noindent
``$\Rightarrow$''. If the complex link $X_t := g^{-1}(t) \cap B_\e$ is contractible, where $g$ denotes a general linear function, then it follows that $\Gamma(g,f) =\emptyset$. This implies that the total Milnor number $\mu(X_t)$ is equal to 
the Milnor number $\mu(X_0)$ of the germ of $X_0$ at $0$. By the non-splitting principle (cf \cite{Le0}, \cite{AC}), this means that $X_t$ has a single singular point and that $\Sing X$ is a non-singular irreducible curve.
Therefore $X$ is an equisingular deformation of a plane curve (which may be reducible).
\end{proof}

 Since the above proof is the same for a higher dimensional hypersurface $(X,0)$, we actually get following statement, which may be compared to an observation by L\^e D.T. \cite[pag. 285, Rem. 2]{Le2}:

\begin{corollary}\label{c:hyp}
 Let $(X, 0)\subset (\bC^n, 0)$, $n\ge 3$,  be a hypersurface germ with 1-dimensional singularity.
 The complex link $\Cl(X,0)$ is contractible if and only if  $X$ is an equisingular deformation of a $(n-2)$-dimensional isolated hypersurface germ.
\fin
\end{corollary}



%


\end{document}